\documentclass[12pt]{amsart}
\usepackage{amsmath, amssymb, amsthm, bm, mathtools, enumitem, caption}

\usepackage[noabbrev,capitalize]{cleveref}
\usepackage{adjustbox}
\crefname{equation}{}{}
\usepackage{ytableau}
\usepackage{graphicx, tikz}
\usepackage[textheight=8.85in, textwidth=7in]{geometry}

\usepackage{microtype}

\newtheorem{theorem}{Theorem}[section]

\newtheorem{corollary}[theorem]{Corollary}

\newtheorem{conjecture}[theorem]{Conjecture}
\newtheorem*{conjecture*}{Conjecture}

\theoremstyle{definition}

\theoremstyle{remark}
\newtheorem*{remark}{Remark}

\numberwithin{equation}{section}

\newcommand{\N}{\mathbb N}

\usepackage[OT2,T1]{fontenc}
\DeclareSymbolFont{cyrletters}{OT2}{wncyr}{m}{n}
\DeclareMathSymbol{\Sha}{\mathalpha}{cyrletters}{"58}

\newcommand{\cM}{\mathcal M}
\newcommand{\cN}{\mathcal N}

\newcommand{\F}{\mathbb F}

\newcommand{\Q}{\mathbb Q}

\newcommand{\Z}{\mathbb Z}

\newcommand\remove[1]{}

\title[ the divisibility and indivisibility
of class numbers of quadratic fields ]{
	A collage of results on the divisibility and indivisibility
	of class numbers of quadratic fields}
\author{Srilakshmi Krishnamoorthy}
\author{Sunil Kumar Pasupulati}
\author{muneeswaran R}
\begin{document}
	\maketitle
	\begin{abstract}
		The investigation of the ideal class group  $Cl_K$ of an algebraic number field $K$ is one of the key subjects of inquiry in algebraic number theory since it encodes a lot of arithmetic information about K. There is a considerable amount of research on many topics linked to quadratic field class groups notably intriguing aspect is the divisibility  of  the class numbers. This article discusses a few recent  results on the divisibility of class numbers and the Izuka conjecture. We also discuss the quantitative aspect of the Izuka conjecture.
	\end{abstract}
	\section{introduction}
	The ideal class group $Cl_K$ of the number field $K$ is one of the essential and mysterious objects in algebraic number theory. It is a finite abelian group and the order $h_K$ is called the class number
	of $K$. The divisibility properties of class numbers provide information to understand the class group's structure. Therefore, studying the divisibility properties of class number of the number fields becomes crucial.
	
	The Cohen-Lenstra heuristics \cite{CL} predicts that the portion of quadratic fields whose class numbers are divisible by $n$ is equal to $\frac{1}{n} $.  
	The notion of the class number was introduced in Gauss's Disquisitiones Arithmeticae, written in 1801. He analyzed the theory of binary quadratic forms, which is closely linked to the theory of quadratic fields.
	He conjectured the following :
	\begin{enumerate}[font={\bfseries},label={C\arabic*.}]
		\item There exist infinitely many  quadratic fields of  class number one.\label{classone}
		\item There exists infinitely many real quadratic fields of the form $\Q(\sqrt{p}), p\equiv1\pmod{4}$ of class number 3. 
		\item There are only finitely many imaginary quadratic fields with class number one.\label{ima}
	\end{enumerate}
	The first part (\ref{classone}) of the above conjecture is famously known as the {\it class number one problem.} 
	He even predicted the complete list of imaginary quadratic fields with class number one. They are $\Q(\sqrt{-d})$  where $d =1,2,3,7,11,19,43,67,163$. 
	The third part (\ref{ima}) of the above conjecture was finally proved independently by Alan Baker \cite{AB71}, Heegner and Harold Stark \cite{SH66} in 1966. 
	Chowla and Friedlander \cite{CF} posed the following conjecture.
	\begin{itemize}
		\item If $d=m^2+1$ is a prime with $m>26,$ then the class number of $\Q(\sqrt{d})$ is greater than 1. 
		\item This conjecture says that there are exactly seven real quadratic fields of the form $\Q(\sqrt{m^2+1})$  with class number 1, and they correspond to $m\in\{1, 2, 4, 6, 10, 14, 26\}$.
	\end{itemize}
	
	In 1973, Weinberger  \cite{MR335471}  established that for infinitely many values of $d$, the class number of the number field $\mathbb{Q}(\sqrt{d^{2n}+4})$ is divisible by $n$ if $n$ is odd, and divisible by $\frac{n}{2}$ if $n$ is even. Later, in 1988, Mollin and Williams [2] proved this conjecture under the assumption of the generalized Riemann hypothesis. Chowla [3] also formulated a conjecture analogous to this for a broader family of real quadratic fields. Specifically, he conjectured the following:
	
	\begin{itemize}
		\item Let $d$ be a square-free integer of the form $d=4m^2+1$ for some positive integer $m$. Then there exist exactly six real quadratic fields $\Q(\sqrt{d})$ of class number 1, viz $d\in\{5, 17, 37, 101, 197, 677\}$.
	\end{itemize}
	Yokoi \cite{MR891892} studied the above conjecture and posed the following conjecture on certain real quadratic fields.
	\begin{itemize}
		\item Let $d$ be a square-free integer of the form $d=m^2+4$ for some positive integer $m$. Then there exist exactly six real quadratic fields $\Q(\sqrt{d})$ of class number 1, viz $d\in\{5, 13, 29, 53, 173, 293\}$.
	\end{itemize}

	The investigation of the divisibility of class numbers is fascinating due to its association with the special values of $L$-functions. Dirichlet introduced the Dirichlet class number formula, which establishes a relationship between the class number of quadratic fields and the special values $L(1,\chi)$. Here, $\chi$ represents a  Dirichlet character, and $L(s,\chi)$ denotes the Dirichlet series associated with $\chi$.
	Birch Swinnerton-Dyer conjecture is an elliptic curve analogue of the analytic class number formula.
	For any elliptic curve defined over $\Q$ of rank zero and square-free conductor $N,$ if $p $ divides  $| E(\Q)| $ under
	certain conditions on the Shafarevich-Tate group $| {\Sha}_D|,$ Krishnamoorthy~\cite[Theorem 2.6]{S16} showed that
	$p \mid |\Sha_D| $ if and only if $ p \mid h_K,$ $K=\Q(\sqrt{-D})$.
	
	\section{Quantitative results on divisibility of class numbers  of  quadratic fields}
	For given integer $n\geq1$ and a positive real number $X$, let $\cN^{+}_n(X)$ (respectively, $\cN^{-}_n(X) $)    be the number of square-free $0<d\le X$ such that  class number of ${\Q(\sqrt{d})}$ (respectively, ${\Q(\sqrt{-d})}$) is divisible by $n$. 
	Cohen and Lenstra \cite{CL} conjectured that the probability of a prime $p$ dividing the class number of real quadratic fields  is equal to 
	\[ 1-\prod_{i=2}^\infty\left(1-\frac{1}{p^i}\right), \] and that for  imaginary quadratic fields is equal to 
	\[ 1-\prod_{i=1}^\infty\left(1-\frac{1}{p^i}\right). \]
	In other words, Cohen and Lenstra \cite{CL} conjectured that $\cN^{+}_n(X) \sim  c_n^+X$ and  $\cN^{-}_n(X) \sim  c_n^- X$, for some $c_n^+,  c_n^-$.
	Murty \cite{RM99} obtained the first quantitative results on $\cN^{+}_n(X)$ and  $\cN^{-}_n(X)$. Murty proved that for any $\epsilon>0$, $\cN^{-}_n(X) \gg X^{\frac{1}{2}+\frac{1}{n}}$ and $\cN^{+}_n(X) \gg X^{{\frac{1}{2n}} -\varepsilon}$. After that, several authors improved the above results. The following theorem is the best-known quantitative result for imaginary quadratic fields due to Soundararajan \cite{KS00}.
	\begin{theorem}[Soundararajan \cite{KS00}]
		For all sufficiently large real number $X$, we have
		\[ \cN_n(X)\gg \begin{cases}
			X^{\frac{1}{2}+\frac{2}{n}-\epsilon}   & \text{if } \ n\equiv 0 \pmod{4}\\
			
			X^{\frac{1}{2}+\frac{3}{n+2}-\epsilon}  & \text{otherwise.}
		\end{cases}\]
		
	\end{theorem}
	Using counting of irreducible polynomials, Chakraborty and Ram Murty \cite{CR03} proved that $\cN^+_3(X)\geq X^{\frac{5}{6}}$.
	Byeon and Koh \cite{DE03} improved Chakraborty and Ram Murty's  bound to  $X^{\frac{7}{8}}$.
	Let $\omega(d_K)$ be the number of distinct prime divisors of $d_K$. Using the results of  \cite{DE03} 
	and  \cite{KM00}, Chattopadhyay proved the following.
	$$\#\left\{K=\Q(\sqrt{d_K})/ \  0<d_K\leq x, \  \omega(d_K) \geq \ell+2 , \ 2^\ell.3\mid h_K\right\}\gg X^{\frac{7}{8}}$$

	\remove{Kalyan Chakraborty and Ram Murty\cite{MR1929021} proved a result for the number of quadratic fields with the class number divisible by $3$.
		\begin{theorem}
			Let $N(x,3)=\#\{\Q(\sqrt{d}): 0<d\leq x$ with $3$ divides class number of $\Q(\sqrt{d})\}$. Then $N(x,3)>> x^{\frac{5}{6}}$.
	\end{theorem}}
	
	Heath-Brown \cite{MR2357319} has improved the lower bound to $\cN^+_3(X) \gg X^{\frac{9}{10}-\epsilon}$ and $\cN^-_3(X) \gg X^{\frac{9}{10}-\epsilon}$.  Yu \cite{MR4159812} constructed a family of complex quadratic fields whose class group has a 3-rank of at least 2. Yu also demonstrated that for sufficiently large $X$, there are $>> X^{\frac{1}{2}-\epsilon}$ such fields with discriminant $-D$, where $D \leq X$. In another work \cite{MR1939135}, Yu established that the number of real quadratic fields with discriminant $\leq X$ and class number divisible by $n$ is $>> X^{\frac{1}{n}-\epsilon}$ for any $\epsilon > 0$ and any odd $n$.
	Siyun Lee, Yoonjin Lee, and Jinjoo Yoo \cite{MR4547469} have made advancements in providing effective lower bounds on the number of imaginary quadratic fields with absolute discriminants less than or equal to $X$ and ideal class groups having a 3-rank of at least one, which is $>> X^{\frac{17}{18}}$. Additionally, they determined that the number of imaginary quadratic fields with a 3-rank of at least two is $>> X^{\frac{2}{3}}$.
	
	\section{Indivisibility of the class numbers of quadratic fields}
	
	Several authors studied the indivisibility of class numbers because this information assists us in predicting the structure of the class groups. In this section, we list a few results on the indivisibility of class numbers.
	Gauss proved that class number of $\Q(\sqrt{-p}), \ p\equiv3\pmod{4}$ is odd. From this, we get infinitely many quadratic fields whose class numbers are not divisible by $2$.  Hartung \cite{MR352040} proved that there exists an infinite family of imaginary quadratic fields whose class numbers are not divisible by $3$.
	Daniel and Fouvry \cite{MR1679790} proved that there are infinitely many odd positive fundamental discriminants $d $ and $ d+4$ such that  $\Q(\sqrt{d})$ and $\Q(\sqrt{d+4})$ both have the class number not divisible by $2$. For given odd prime $p$, Kohnen and Ono  \cite{MR1666783} estimated the number of  imaginary quadratic fields whose discriminants are  less than $X$ and the class number is not divisible by $p$   to be 
	$ \gg \frac{\sqrt{x}}{\log{x}}$.
	For any square free integer $t$,   Byeon \cite{MR2073286} proved that there exist infinitely many $d>0$ such that  $3$ does not divide the class numbers 
	of $\Q(\sqrt{d})$ and $\Q(\sqrt{td})$. In fact, he also proved that the set  
	\[ \left\{  d : 3 \nmid  h_{\Q(\sqrt{d})} \  \text{and} \ 3 \nmid h_{\Q(\sqrt{td})} \right\} \]
	has a positive density. Chattopadhyay and Saikia \cite{jaitra&anupam} proved a result on the density of discriminant such that $3$ doesn't divide the class numbers of $\Q(\sqrt{d}), \ \Q(\sqrt{d+t})$ for $t\equiv0\pmod{4}$.
	Lee and  Lee \cite{MR3471189} proved that the density of real quadratic function fields with class numbers not divisible by $\ell$ is $\frac{\ell-2}{\ell-1}$. Wang  \cite{wang} established the existence of infinitely many  quadratic fields whose class numbers are not divisible by $n$ for $n>2$. 
	Wiles \cite{MR3404031} established the existence of imaginary quadratic fields with prescribed local data whose class numbers are indivisible by a given odd prime $\ell$.
	\begin{theorem}
		Let $\ell\geq5$ be prime and let $S_0, \ S_+, \ S_-$ be finite disjoint sets of distinct odd primes not contain  $\ell$ such that the following are true:
		\begin{itemize}
			\item $S_0$ does not contain any primes which are 
			congruent to $1$ modulo  ${\ell}$.
			\item $S_+$ does not contain any primes which are congruent to $-1$ modulo  ${\ell}$.
			\item $S_-$ does not contain any primes which are congruent to $1$ modulo  ${\ell}$ and $-1$ modulo $4$.
		\end{itemize}
		Then there exists a negative fundamental discriminant $d$ such that $\ell$ does not divide the class number of $\Q(\sqrt{d})$. Moreover, the 
		extension $Q(\sqrt{d})$ splits at every prime in $S_+,$ is inert at every prime at $S_-,$ and ramifies at every prime in $S_0$. 
	\end{theorem}
	The following indivisibility result for function fields was proved by Pacelli, Pitiwan, and Rosen \cite{MR2847264}.
	\begin{theorem}
		Let $m>1$ be any integer and $\ell$ be an odd prime divisor of $m$. Write $m=\ell^tm_1$ for integers $t$ and $m_1$ with $d\nmid m_1$. Let $m_0$ be the square-free part of $m_1,$ and assume that $q$ is sufficiently large with $q\equiv1\pmod{m}$ and $q\equiv-1\pmod{\ell}. $ Then there is infinitely many function fields $K$ of degree $m$ over $\mathbb{F}_q(T)$ with $\ell\nmid h_K$.
	\end{theorem}
	
	\section{Divisibility of class numbers  of  an infinite family $K_{x,y,n,\mu}$}\label{sect4}
	
	In 1955, Ankeny and Chowla \cite{MR85301} proved that there are infinitely many quadratic imaginary fields for any natural number $n$, each having a class number divisible by $n$. Later, in 1970, Yamamoto \cite{MR266898}  established the existence of infinitely many real quadratic number fields with ideal class numbers that are multiples of $n$. Additionally, Yamamoto demonstrated that there are infinitely many imaginary quadratic number fields whose ideal class group contains a subgroup isomorphic to the direct product of two cyclic groups of order $n$.
	
	In the study of divisibility of class numbers, the family $K_{x,y,n,\mu}:=Q(\sqrt{x^2-\mu y^n})$  with the conditions $\gcd(x,y) = 1,y> 1, \mu \in \{1,2,4\}$ and $x^2<\mu y^n$  grabbed many eyes. In 1922, T. Nagell \cite{TN22}  proved that for an odd integer $n$, the class number of ${K_{x,y,n,1}}$ is divisible by $n$ if $t$ is odd, $\gcd(t,x) =1$ and $q \mid x,$ $q^2 \nmid x$ for all prime divisors $q$ of $n$.
	Using the affine points on the Fermat curve $x^p+y^p=1$ over the imaginary quadratic field $\Q(\sqrt{1-4y^n})$,	
	Gross and Rohrlich indicated the proof of the class numbers of $K_{1,y,n,1}$ are divisible by $n$ for any odd prime $n >3$ and $ y>1$. Later, using the methods of basic algebraic number theory,  Louboutin \cite{SL}  extended the same result for any positive odd integer  $n$  and proved the following result on the divisibility of class number of $K_{1,y,n,1}$ for $y>2$.  
	He also proved that if at least one of the prime divisors of an odd integer $y >3 $ is equal to $3 \mod 4$, then for any positive integer $n$,  the class number of $K_{1,y,n,1}$ is divisible by $n$. Murty \cite{RM99} proved that  the class number of $K_{1,y,n,1}$ is divisible by $n$ if $1-y^n$ is square-free.
	Chakraborty, Hoque, Kishi, and Pandey \cite{MR3734353} proved the class number divisibility of $k_{x,y,n,1}$ by $n$ under certain mild assumptions.
	Hoque and Chakraborty \cite{MR4010381} studied the  
	3-divisibility of the class numbers of $K_{1,y,3,2}$ 
	and proved the following theorem.
	\begin{theorem}
		The class number of the imaginary quadratic field $K_{1,y,3,2}$ is divisible by $3$ for any $m>1$.
	\end{theorem}
	Krishnamoorthy and Pasupulati \cite{Sunil} extended the above result to any odd prime with mild assumptions using results on the number of solutions of certain Diophantine equations.
	\begin{theorem}\label{sunil}
		The class number of the imaginary quadratic field $K_{1,m,n,2}$ is divisible by $p$ for any odd prime $p$ and $m$ is any odd prime power.
	\end{theorem}
	Using the above theorem, they proved the following corollary.
	\begin{corollary}\label{biquad}
		For any odd prime $p\geq3$, there exist infinitely many imaginary bi-quadratic fields whose class number is divisible by $p$.
	\end{corollary}
	Later Krishnamoorthy and Muneeswaran \cite{KM21} generalized     \Cref{sunil},  and proved the following three results.
	\begin{theorem}\label{munees}
		Let $n\geq3$ be an odd number with prime factorization $n=p_1^{k_1}p_2^{k_2}...p_n^{k_n}$. Consider any odd integer $m>\max\{2^{\frac{p_i-2}{p_i^{k_i-1}}}: 1\leq i\leq n\}$. The class number of $K_{1,m,n,2}$ is divisible by $n$.
	\end{theorem}
	The values of $n$ such that the class number of $K_{1,m,n,2}$ is divisible by $n$ were analyzed. In that attempt, the following result was also established.
	\begin{theorem}
		Let $m\geq3$ be an odd integer. \begin{enumerate}
			\item  There exists a natural number $r$ such that for any odd number $n\geq3$ coprime to $r$, the class number of $K_{1,m,n,2}$ is divisible by the square-free part of $n$.
			\item  If $n,m\geq3$ are odd integer such that $1-2m^n$ is a square-free integer, then the class number of $K_{1,m,n,2}$ is divisible by the square-free part of $n$.
		\end{enumerate}
	\end{theorem}
	In that same work, a similar result connected with twin primes was also proved.
	\begin{theorem}
		If $p_1, \ p_2$ is any pair of twin primes, then at least one of the $p_i$ divides the class number of $K_{1,m,p_i,2}$.
	\end{theorem}
	\begin{remark}
		The twin prime conjecture states that there are infinitely many twin primes. If the twin prime conjecture is true, then for each pair of twin primes, we can have at least one prime $p$ which divides the class number of $K_{1,y,p,2}$, and this also shows that there are infinitely many primes $p$ which divide the class number of $K_{1,m,p,2}$.
	\end{remark}
	
	Hoque \cite{MR4270672} proved that under certain conditions on $a,p,n$, the class group of the imaginary quadratic field $ K_{a,p,n,2}$ has a subgroup isomorphic to $\Z/n\Z$.
	\begin{theorem}
		Let $a\geq1, n \geq3$ be odd integers $p$ be a prime such that $(a,p)=1$ and $a^2<4p^n$. Suppose $-d$ is the square-free part of $a^2-4p^n$. For $a\neq1,$ assume one of the following conditions holds:
		\begin{itemize}
			\item $a\not\equiv\pm b\pmod{\ell}, \ \forall \  b\mid a, \ b\neq a$ and $\forall \ \ell\mid n, \ \ell$ is prime and $d\neq3$.
			\item $2^{\ell-1}a\not\equiv b^{\ell}\pmod{d} \ \forall \  b\mid a, \ b\neq a$ and $\forall \ \ell\mid n, \ \ell$ is prime.
		\end{itemize}
		Then except for $(a,p,n)\in\{(5, 2, 3),(5, 2, 9),(11, 2, 5),(13, 2, 7)\},$ the class group of $\Q(\sqrt{-d})$ has a subgroup isomorphic to $\Z/n\Z$.
		
	\end{theorem}

	\begin{theorem}
		Let $m\geq3$ be an odd integer with distinct odd primes with $q^2<p^n$. Let $d$ be the square-free part of $q^2-p^n$. Assume that $1\not\equiv\pm1\pmod{|d|}$. Moreover, we assume $p^{\frac{n}{3}}\neq\frac{2q+1}{3}, \ \frac{q^2+3}{2}$ whenever both $d\equiv1\pmod{4}$ and $3\mid n$. Then the class number of $K_{q,p,n,1}$ is divisible by $n$.
	\end{theorem}

	Azizul Hoque and Kalyan Chakraborty \cite{MR4181805} proved the following result.
	\begin{theorem}
		Let $p$ and $q$ be distinct odd primes and $n\geq3$ an odd integer with the property that $3q^{\frac{n}{3}}\neq n+2$ whenever $3\mid n$. The class number of $K_{p,q,n,2}$ is divisible by $n$. Moreover, there are infinitely many imaginary quadratic fields with discriminant of the form $p^2-2q^n$.
	\end{theorem}
	
	\begin{remark}
		The indivisibility of the class numbers of $K_{x,y,n,\mu}$ is very less known.
	\end{remark}
	
	To explore the details regarding the divisibility of class numbers of the quadratic  fields beyond the family $K_{x,y,n,\mu}$, I recommend referring to the survey article authored by Bhand and Murty \cite{MR4221215}. This article provides comprehensive information on the topic and will serve as a valuable resource for interested readers.
	
	\section{Towards Iizuka's conjecture}
	
	In this section,  first, we discuss the $3$-divisibility of class numbers of quadratic fields and move towards the recent developments on Iizuka’s conjecture.
	
	
	\subsection{$3$-divisibility of the class number of the quadratic fields}
	Kishi and  Miyake \cite{KM00}  provided a parametric family of the quadratic fields with the class numbers divisible by $3$. They considered $u,w\in\Z$  such that 
	\begin{enumerate}
		\item  The integral  polynomial  $g(T) =T^3-uwT-u^2 \in \Z[T]$ is irreducible  over $\Q$.
		\item  $\gcd(u, w) = 1$.
		\item $d := 4uw^3-27u^2 $ is not a perfect square in $\Z$. 
		\item  One of the following conditions holds:-
		\begin{enumerate}
			\item $3\nmid w$.
			\item $3\mid w, \ uw\not\equiv 3\pmod{9}, \ u\equiv(1\pm w)\pmod{9}$.
			\item  $3\mid w, \ uw\equiv 3\pmod{9}, \ u\equiv(1\pm w)\pmod{27}$.
		\end{enumerate}
	\end{enumerate}

	\begin{theorem}[ Kishi and  Miyake \cite{KM00}]
		
		The quadratic fields whose class numbers are multiple of   $3$ are of the form $\Q(\sqrt{4uw^3-27u^2})$ for some $u,w$ satisfying  the above conditions $(1)-(4)$.
	\end{theorem}
	
	\begin{remark}
		Li and  Zhang  \cite{zhang} characterized  $3$-divisibility of class numbers of quadratic extensions of the function fields.
	\end{remark}
	Erickson et al. \cite{Carl} parameterized the quadratic fields with  $3$-ranks of the class group are least $2$. 
	\begin{theorem}
		Let  $w \equiv \pm 1 \pmod6$, and let $c$ be any integer with $c\equiv w \pmod 6$. Then  the class group of 
		\[
		\Q\left( \sqrt{c(w^2+ 18cw+ 108c^2)(4w^3-27cw^2-486c^2w-2916c^3)} \right) 
		\]
		has $3$-rank at least $2$. 
	\end{theorem}

	\subsubsection{Simultaneous $3$-divisibility of the class number of quadratic fields} 
	We start with one of the old theorems, Scholz's reflection principle, which discusses the simultaneous $3$-divisibility of 
	real and imaginary quadratic fields. In 1932, Scholz \cite{SC32} proved the following theorem.
	\begin{theorem}\label{schloz}
		Let $d$ be a positive integer and $r,s$ be $3$-ranks of class groups of $\Q(\sqrt{d}), \Q(\sqrt{-3d})$ respectively. Then 
		$ r\leq s\leq r+1.$
	\end{theorem}
	Recently several authors have shown interest in studying the simultaneous $3$-divisibility of the class number of quadratic fields. We will list a few here. 
	Komatsu \cite{K01} explicitly constructed infinitely many pairs of quadratic fields $\Q(\sqrt{d}), \Q(\sqrt{-d})$  whose class numbers are simultaneously divisible by $3$. Later, for given an integer $m\in\Z$ with $m\neq0$, Komatsu \cite{K02} proved that there exists an infinite family of pair of quadratic fields $\Q(\sqrt{d})$ and $\Q(\sqrt{md})$ whose class numbers are divisible by $3$.
	Ito \cite{AI13} established the existence of an infinite family of pairs of quadratic fields $\Q(\sqrt{m_1D})$ and $\Q(\sqrt{m_2D})$ whose class numbers are both divisible by $3$ or both indivisible by $3$. 
	Iizuka, Konomi, and Nakano \cite{YYN} constructed pairs of quadratic fields whose class numbers are divisible by  $3,5$, and  $7$ by associating the problem to the study of points on elliptic curves. Kalita and  Saikia \cite{MR4292551} proved that the class numbers of the pairs of quadratic fields $\Q(\sqrt{p^{12\ell+2}-4})$ and $\Q\left(\sqrt{\frac{4-p^{12\ell+2}}{3}}\right)$ simultaneously divisible by $3$, whenever  $p\equiv\pm 4 \pmod{9}$ is a prime and $\ell\geq1$ an integer.
	\remove{
		\begin{remark}
			If $D$ and $d$ denote the discriminant of $\Q(\sqrt{p^{12\ell+2}-4})$ and $\Q\left(\sqrt{\frac{4-p^{12\ell+2}}{3}}\right),$ then $d\mid D$.
		\end{remark}
	}
	
	\subsection{Iizuka's Conjecture}
	
	Iizuka \cite{Iizuka} proved the following two results on $3$-divisibility of the class numbers of quadratic fields.
	\begin{theorem}\label{iizuka1}
		For any nonzero integer $t,$ the class number of $\Q\left(\sqrt{t(432t^2+36t+1)} \right)$ is divisible by $3$.
	\end{theorem}
	\begin{theorem}\label{iizuka2}
		For any nonzero integer $t,$ the class number of $\Q(\sqrt{3(108t^3-1)})$ is divisible by $3$.
		
	\end{theorem}
	Using \Cref{iizuka1} and \Cref{iizuka2}, he proved the following interesting theorem.
	\begin{theorem}\label{Iizuka}
		There exist infinitely many imaginary pairs of quadratic fields $\Q(\sqrt{d})$,  $\Q(\sqrt{d+1})$ with $d\in\Z$ whose class numbers are simultaneously divisible by $3$. 
	\end{theorem}

	Based on the above theorem, Iizuka made the following conjecture.
	\begin{conjecture}
		Given any $n\in\N$ and a prime $p,$ there exist infinitely many $d\in\Z$ such that the class numbers of $\Q(\sqrt{d}), \ \Q(\sqrt{d+1}),\cdots,\Q(\sqrt{d+n})$ are divisible by $p$.
	\end{conjecture}
	Several authors showed interest in  Iizuka's conjecture and proved partial results. Till now, this conjecture is proved only for $n=1$.
	Krishnamoorthy and  Pasupulati \cite{Sunil} proved the particular case of this conjecture.
	\begin{theorem}
		Given any prime $p,$ there exists infinitely many $d\in\Z$ such that class numbers of $\Q(\sqrt{d}), \ \Q(\sqrt{d+1})$ are divisible by $p$.
	\end{theorem}
	Given any odd $n$, using \Cref{munees} Krishnamoorthy and Muneeswaran \cite{KM21} produced infinitely many $d$ with $\Q(\sqrt{d}), \ \Q(\sqrt{d+1})$ is divisible by $n$.
	Chattopadhyay and Muthukrishnan \cite{triple} proved that the class numbers of the following triples of quadratic fields are simultaneously divisible by 3. 
	\begin{theorem}
		Let $k\geq1$ be a cube free integer with $k\equiv1\pmod{9}$ and $\mathrm{gcd}(k, 7.571)=1$. Then there exists infinitely many triples of imaginary quadratic fields $\Q(\sqrt{d})$,$ \Q(\sqrt{d+1})$, 
		$ \Q(\sqrt{d+k^2})$ with $d\in\Z$ whose class numbers are simultaneously divisible by $3$.
	\end{theorem}
	Hoque \cite{hoque} proved a weaker version of Iizuka's conjecture.
	\begin{theorem}
		For any odd positive integer $n$ and an odd prime $p$ there are infinitely many quadruples of imaginary quadratic fields $\Q(\sqrt{d}), \ \Q(\sqrt{d+1}), \ \Q(\sqrt{d+4}), \ \Q(\sqrt{d+4p^2})$ with $d\in\Z$ whose class numbers are divisible by $n$.
	\end{theorem}
	Chakraborty and Hoque \cite{arxivhoque}  produced an infinite family of certain tuples of imaginary quadratic fields whose class numbers are divisible by $n$. This result is to appear in the Ramanujan Journal. 
	\begin{theorem}
		There exists an infinite family of certain tuples of imaginary quadratic fields of the form $$\Q\left(\sqrt{d}\right), \ \Q\left(\sqrt{d+1}\right), \ \Q\left(\sqrt{4d+1}\right), \ \Q\left(\sqrt{2d+4}\right), \ \Q\left(\sqrt{2d+16}\right),..., \Q\left(\sqrt{2d+4^t} \right)$$ with $d\in\Z, \  1\leq4^t\leq 2|d|$  whose class numbers are all divisible by $n$.
	\end{theorem}
	
	Many authors proved results similar to \Cref{Iizuka}. For example,  Xie and  Fai \cite{xie} proved the following result using parametrization of quadratic fields whose class numbers are divisible by $3$.
	\begin{theorem}
		For arbitrary positive integer $n,$ there exists infinitely many pairs of quadratic fields $\Q\left(\sqrt{d}\right), \ \Q\left(\sqrt{d+n}\right)$ with $d\in\Z$ their class number simultaneously divisible by $3$.
	\end{theorem}
	
	\section{Quantitative aspects of Iizuka conjecture}

	For given $n$,   $\cM_n^{-}(X)$ to be the number of
	$d\leq X$ such that the class numbers of both the imaginary quadratic field  ${\Q(\sqrt{-d})}, \ {\Q(\sqrt{-d +1})}$ are divisible by $n$.
	\begin{theorem}
		For large $X$,
		\[\cM_n^{-}(X)\gg (X/4)^{\frac{1}{2n}+\frac{2}{n^2}- \frac{\epsilon}{n}}.
		\]
	\end{theorem}
	
	\begin{proof}
		Let $n$ be an odd integer that divides the class number of the imaginary quadratic field ${\Q(\sqrt{-d})}$. By \cite[Theorem 1]{SL}, $n$ divides the the class number of the imaginary quadratic field $\Q(\sqrt{1-4d^n})$. If an odd integer  $n$  divides  $h_{\Q(\sqrt{-d})}$ then
		$n$ divides  the class number of consecutive  imaginary quadratic fields   ${\Q(\sqrt{-D})} $ and $ {\Q(\sqrt{-D +1})}$, where $D=4d^n$. Therefore 
		\begin{align*}
			\cM_n^{-}(X)= &\left\{  D\leq X : 3\ \text{ divide both } \  h_{\Q(\sqrt{-D})}, h_{\Q(\sqrt{-D +1})} \right\}. \\
			&\geq \cN_n\left((X/4)^{\frac{1}{n}}\right)\\
			&\gg (X/4)^{\frac{1}{n}* \left(\frac{1}{2}+\frac{2}{n}-\varepsilon \right) }\\
			&=(X/4)^{\frac{1}{2n}+\frac{2}{n^2}-\frac{\varepsilon}{n}}.
		\end{align*} 
	\end{proof}
	
	\section{Concluding Remarks}
	Cornell \cite{MR564924} proved that any finite abelian group is a subgroup of the class group of some cyclotomic fields. But he has not said anything about the number of cyclotomic fields. Mishra et al. \cite{MR4270784} proved the following result which is similar to Cornell's result. But in real cyclotomic fields, which is also told about the number of real cyclotomic fields. 
	\begin{theorem}
		Given any finite abelian group $G$ there exists infinitely many real cyclotomic fields $\Q(\zeta_n)^+$ whose class group has a subgroup isomorphic to $G$. 
	\end{theorem}
	\begin{remark}
		Analogous to \cref{Iizuka}, we can ask for a given finite abelian group $G$, whether it is possible to find $n$ such that $G$ is a subgroup of both the class groups of real cyclotomic fields $\Q(\zeta_n)^+$ and $\Q(\zeta_{n+1})^+$.
	\end{remark}
	
	\subsection{Function fields}
	Let $k=\F_q(t)$ be the rational function field over the field $\F_q$, where $q$ is a power of some rational prime $p$. Consider the polynomial $f(X)=X^3-uwX-u^2, u,w\in \F_q[t]$, where $u$ and $w$ are relatively prime and $\deg u<3\deg w $ or  $3\mid \deg u$.  Assume that $f(X)$ is irreducible over $k$. Let $d=4u^3w^3-27u^4$ be the discriminant of $f(X)$ and $g(X)=X^2+u^2X+u^3w^3+u^4$. For each global function field $K$, we denote by $h(K)$ the number of divisor classes of degree $0$ and call it the divisor class number of $K$ for brevity. 
	
	Li and Zhang \cite{zhang}  proved the following, 
	\begin{enumerate}
		\item  For  $p>3$,  if $d$ is not a square in $k$, then $3$ divides the divisor class number of  $ k(\sqrt{d})$.  Conversely, every quadratic function field whose divisor class number is divisible by $3$ is given in this way by some $u$ and $w$. 
		\item When $p=2$, if $g(X)$ is irreducible over $k$, let $K$ be the quadratic function field generated by the roots of $g(X)=0$ .Then $3$ divides the divisor class number of  $ h(K)$. Conversely, every quadratic function field whose divisor class number is divisible by $3$ is given in this way by suitable $u$ and $w$.  
		\end {enumerate}
		
		\begin{remark}
			Iizuka used the characterization of the quadratic fields whose class numbers are divisible by $3$ to construct infinitely many $d$ such that the class numbers of $\Q(\sqrt{d}), \  \Q(\sqrt{d+1})$ are divisible by $3$. Similarly, one can use the characterization of function fields to prove analogous results on  $3$-divisibility  of divisor class numbers of the quadratic extensions of functional fields. 
		\end{remark}

		\bibliographystyle{amsplain}

	\end{document}